\documentclass{amsart}
\usepackage[utf8]{inputenc}

\usepackage{tikz} \usetikzlibrary{calc,matrix,patterns}
\usepackage{subfig}

\newcommand{\diam}{\mathop{\mathrm{diam}}}

\newtheorem{thm}{Theorem}
\newtheorem{pro}[thm]{Proposition}

\newtheorem{lem}[thm]{Lemma}
\newtheorem{cor}[thm]{Corollary}

\newtheorem{qes}{Question}
\theoremstyle{remark}
\newtheorem{rmk}[thm]{Remark}
\newtheorem{exm}{Example}

\newcommand{\cc}{{\,\mathcal C}}

\newcommand{\lc}{{\,\mathcal L}}

\newcommand{\pc}{{\,\mathcal P}}
\newcommand{\rc}{{\,\mathcal R}}

\newcommand{\NN}{{\mathbb N}}

\newcommand{\RR}{{\mathbb R}}

\newcommand{\ZZ}{{\mathbb Z}}

\newcommand{\lra}[1]{\langle #1 \rangle}

\title{On the partitions with Sturmian-like refinements}
\date{\today}
\begin{document}
\maketitle
\author{
M. Kupsa
\footnote{\tiny Institute of Information Theory and
  Automation, The Academy of Sciences of the Czech Republic, Prague 8, CZ-18208}
\footnote{\tiny Faculty of Information Technology,
  Czech Technical University in Prague, Prague 6, CZ-16000.},
Š. Starosta$^2$ 
}

\begin{abstract}
In the dynamics of a rotation of the unit circle by an irrational angle $\alpha\in(0,1)$, we study the evolution of partitions whose atoms are finite unions of left-closed right-open intervals with endpoints lying on the past trajectory of the point $0$. 
Unlike the standard framework, we focus on partitions whose atoms are disconnected sets.
We show that the refinements of these partitions eventually coincide with the refinements of a preimage of the Sturmian partition, which consists of two intervals $[0,1-\alpha)$ and $[1-\alpha,1)$. 
In particular, the refinements of the partitions eventually consist of connected sets, i.e., intervals. We reformulate this result in terms of Sturmian subshifts: we show that for every non-trivial factor mapping from a one-sided Sturmian subshift, satisfying a mild technical assumption, the sliding block code of sufficiently large length induced by the mapping is injective. 
\end{abstract}

\section{Main results}
\label{sec:main-results}
The dynamics given by a mapping $T$ from a set $X$ to itself is often
described using the coding of orbits of points with respect to some
specific partition $\rc$ of $X$. This standard approach of symbolic
dynamics involves an analysis of the evolution of the partition $\rc$
with respect to $T$, where by the evolution we mean the refining
sequence of partitions $\rc^n$ for $n\geq 1$, defined as follows:
\[
\mathcal{R}^{n}=\left\{\bigcap_{k=0}^{n-1}T^{-k}R_k \ \middle| \
  R_0,\ldots, R_{n-1} \in \rc\right\} \setminus \{\emptyset\}.
\]
The partition $\rc^n$ is the common refinement of $\rc$ and its preimages $T^{-j}\rc$ for all integers $j$ such that $1 \leq j<n$.
We call $\rc^n$ the \emph{$n$-th refinement of $\rc$}.

Focusing on the dynamics of an irrational rotation of the unit circle,
the most studied partitions are those inducing Sturmian
sequences. Given an irrational $\alpha\in (0,1)$, the unit circle is
the factor group $X=\RR/\ZZ$, represented by the fundamental domain
$[0,1)$, and the rotation $T$ is the transformation of $X$ given
by the formula $T(x)=(x+\alpha) \bmod{1}$. The partition inducing
Sturmian sequences consists of two intervals $P_{0}=[0,1-\alpha)$ and
$P_{1}=[1-\alpha,1)$.  We call this partition {\em Sturmian} and
denote it by $\mathcal{P}$, i.e., $\mathcal{P} = \{ P_0, P_1 \}$.

The evolution of the partition $\pc$ is closely related to
combinatorial and other properties of the Sturmian sequence obtained
as a coding of the orbit of the point $0$ with respect to $\pc$ (for
detailed study of Sturmian sequences see \cite{Fo02}, \cite{Ku03-2} or
\cite{HM40}). It is well known that the refinement $\pc^n$, for $n\in\NN$,
consists of $n+1$ intervals in which the points $T^{-k}(0)$, for $0\leq
k\leq n$, divide the unit circle. The Three lengths theorem, due to S\'{o}s
(\cite{So58}), claims that these intervals are of two or three
lengths. The theorem also describes these lengths in terms of
convergents of $\alpha$.
Because of the trivial identity $(\pc^m)^n=\pc^{m+n-1}$, the evolution
of any refinement $\pc^m$ is covered by the mentioned results as well.

There is much less known about the evolution of other partitions.
Combinatorial results for coding with respect to two-interval or
finite-interval partitions with arbitrary endpoints were obtained
in \cite{Di98}, \cite{Al96} and \cite{AB98}. 
\\

In this paper, we would like to introduce another class of partitions
whose evolution can be surprisingly well described. The class consists
of all partitions whose elements are finite unions of right-closed
left-open intervals with endpoints from the set of preimages of
the zero $T^{-i}(0)$ for $i\in\NN$. Although the endpoints are chosen in a standard manner from the past trajectory of the point zero, the partitions stand outside the classical framework of coding of rotations because, in our case, the atoms of the partitions are usually disconnected sets. Partitions from this class are
closely related to the partition $\pc$, namely $\rc$ belongs to the
class if and only if $\rc$ is rougher than $\pc^n$ for some $n\in\NN$ (a complete definition follows in Preliminaries).
In other words, a partition from the class consists of the sets that
belongs to the algebra of sets generated by $\pc^n$, for some $n$ (the
sets are $\pc^n$-measurable). We call these partitions {\em Sturmian-measurable} throughout the paper.

Since the partition $\pc$ and its preimages $T^{-j}\pc$, $j\in\NN$,
generate the $\sigma$-field of Borel subsets of the unit interval, the
class of all Sturmian-measurable partitions has the following
interesting property; it is a dense set among all Borel partitions
with respect to Rokhlin distance or entropy distance.  Our main result
shows that the refinements of any partition from the class coincide
with the refinements of some preimage of $\pc$.

In the following theorem, we introduce the main result of our paper. It concerns the refinements of Sturmian-measurable partitions that are non-trivial, i.e., consisting of at least two sets.   
\begin{thm}\label{thm:mainresult}
  Let $n\in\NN$.  If $\rc$ is a non-trivial partition rougher than
  $\pc^n$, then
$$ \rc^k=T^{-\ell} \pc^{m} =\left(T^{-\ell}\pc\right)^{m},\qquad\text{ for some } k,\ell,m\in\NN \text{ such that } \ell<n.$$
In other words, $\rc^k$ is the partition of the unit circle into a
union of right-closed left-open intervals whose endpoints are the
preimages of zero $T^{-i}(0)$ for $i \in \NN$ such that $\ell \leq i\leq \ell+m$.
\end{thm}

Let us notice that whenever the partition $\rc^k$ equals
$\left(T^{-\ell}\pc\right)^{m}$, then for $i\in\NN$ every higher
refinement $\rc^{k+i}$ equals the higher refinement
$\left(T^{-\ell}\pc\right)^{m+i}$. In this case, the sequences
$(\rc^k)_{k\in\NN}$ and $((T^{-\ell}\pc)^m)_{m\in\NN}$ have the same
tail.

The least $k$ such that the partition $\rc^k$ is of the form described
in Theorem \ref{thm:mainresult} strongly depends on $\rc$. In the next
theorem, we provide an upper bound for the power $k$ in terms of
convergents of $\alpha$.  The continued fraction expansion of $\alpha$
is the following:
$$
\alpha=[c_1,c_2,c_3,\dots]=\cfrac{1}{c_1+\cfrac1{c_2+\cfrac1{c_3+\ldots}}},\qquad
c_i\in\NN \setminus \{0\}.
$$
The convergents of $\alpha$ are then $\frac{p_k}{q_k}$ where
$p_k=c_kp_{k-1}+p_{k-2}$, $p_0=0,p_1=1$ and $q_k=c_kq_{k-1}+q_{k-2}$,
$q_0=1,q_1=c_1$.  Denote $r_k=q_k+q_{k-1}$ for $k\geq 1$ and $r_0=1$.

\begin{thm}\label{thm:precise}
  Let $\rc$ be a Sturmian-measurable non-trivial partition. Let $\ell$
  be the largest positive integer and $n$ the least positive integer
  such that $\rc$ is rougher than $T^{-\ell}\pc^n$.  If $k\in\NN$ such
  that $r_{k-1}\leq n<r_k$, then
$$\rc^{r_{k+3}+2r_k-n-2}=T^{-\ell}\pc^{r_{k+3}+2r_k-3}=\left(T^{-\ell}\pc\right)^{r_{k+3}+2r_k-3}.$$
\end{thm}

\begin{rmk}
  The numbers $\ell$ and $n$ in the statement of the last theorem are
  unambiguous.  We will see this fact at the end of Section
  \ref{sec:sturm-measurable}.
\end{rmk}

Let us notice that the theorem does not give the optimal answer when we
look for the minimal $m$ such that $\rc^m$ contains no disconnected
set. It is neither optimal when searching for the minimal $m$ such
that $\rc^m$ is of the form $T^{-i}\pc^j$ for $ i,j\in\NN$. Indeed, when
$\rc$ equals $\pc^n$ for some $n$, the optimal answer for both
problems is $m=1$, whereas the theorem suggests a number greater than
$r_{k+3}+r_k-1$.

In Section \ref{sec:subshifts}, we rephrase the results in terms of
symbolic dynamics (see Proposition \ref{pro:sturm-sliding}), namely in terms of injectivity of a sliding block code of finite length induced by a factor mapping from 
a
Sturmian subshift to another subshift. We show that for every non-trivial factor mapping from a Sturmian subshift, satisfying a mild technical assumption, the sliding block code of sufficiently large length induced by the mapping is injective. 
An immediate consequence of this result is the fact, that the factor mapping itself is injective.
Let us remark that the injectivity of the factor mapping can be proved by making slight changes in the proof of Cantor primarility of a two-sided Sturmian subshift given in (\cite{Du00}), see Proposition \ref{pro:Durand}. Nevertheless, we did not find a way to adapt this proof to pass from the injectivity of the factor mapping to the injectivity of the induced sliding block codes in the case of Sturmian subshift. 

In Examples \ref{exm:symbolic} and \ref{exm:symbolic_minimal} we show that, in general, the injectivity of a factor mapping from a subshift to another subshift does not imply the existence of an injective sliding block code of finite length.

The interpretation of the main result in the frame of symbolic dynamics invokes its own related problems and open questions, which are discussed in Section \ref{sec:open-problems}.   



\section{Preliminaries}
\label{sec:intro}

Let $X$ be a set and $T:X\to X$ be a mapping on it. A \emph{partition}
$\rc$ of the space $X$ is a set of non-empty pairwise disjoint sets
from $X$ such that they cover the whole set $X$, i.e., $X=\bigcup_{R
  \in \rc} R$. We say that a partition $\rc$ \emph{is finer than} a
partition $\rc'$ (or equivalently we say that $\rc'$ \emph{is rougher
  than} $\rc$) if every $R\in\rc$ is a subset of a set $R'\in\rc'$. In
other words, every set from $\rc'$ is a union of sets from $\rc$. This
relation, denoted by $\rc>\rc'$, forms a lattice structure on the set
of all partitions of $X$. The \emph{supremum} of two partitions is
denoted by $\vee$, it is also called the \emph{join}, and for two partitions
$\rc$ and $\rc'$ is defined as follows:
$$
\rc\vee\rc'=\{R\cap R' \mid R\in\rc,R'\in\rc'\}\setminus\{\emptyset\}.
$$
It is readily seen that $\rc^n=\bigvee^{n-1}_{i=0}T^{-i}\rc$ for
every partition $\rc$ and every $n\geq 1$.

Let a partition $\rc$ be labeled by indices forming a set $\Sigma$, i.e.,
$\rc=\{R_i \mid i\in \Sigma\}$. The \emph{labeling} of the partition
$\rc$ can be described by the mapping $\phi_\rc:X\to\Sigma$, where
$\phi_{\rc}(x)=i$ if $x\in R_i$. The sequence
$(\phi_{\rc}(T^ix))^n_{i=0}$ is called the \emph{$\rc$-name} of $x$ of
length $n$.  We get
\[
\rc^{n}=\left\{ R_{u} \mid
  u\in\Sigma^{n}\right\}\setminus\{\emptyset\} ,\qquad\text{where
}R_{u}=\bigcap_{k=0}^{n-1}T^{-k}R_{u_{k}}.
\]
In other words, the partition $\rc^n$ is the partition induced by the
$\rc$-names of the points from $X$ of length $n$, i.e., two points
from $X$ are in the same set from $\rc^n$ if and only if they have the
same $\rc$-name of length $n$.

For a given non-empty set $A \subset X$ we denote the restriction to
$A$ of a partition $\rc$ by $\rc|A$, i.e., $\rc|A=\{R\cap A \mid
R\in\rc\}\setminus \{ \emptyset \}$. For $i\leq j$, denote the
following family of sets:
$$\Lambda(A,i,j)=\{T^{-m}A \mid i\leq m<j\}.$$
If the sets in $\Lambda(A,i,j)$ are pairwise disjoint, we call
$\Lambda(A,i,j)$ a \emph{Rokhlin tower}(or simply a tower). The set
$T^{-i}A$ is called the \emph{base of the Rokhlin tower}, $T^{-(j-1)}A$ is
called the \emph{top} and $j-i$ is the \emph{height} of the Rokhlin
tower.  The set $T^{-k}A$, $i\leq k<j$, is referred to as $(k-i)$-th
\emph{level} of the tower.  We say that a word $u=u_0u_1\ldots
u_{j-i-1}\in\Sigma^{j-i}$ is the \emph{$\rc$-code of the tower} if
$$T^{-(j-1)+k}A\subset R_{u_k},\qquad \text{ for every } k \text{ such that } 0\leq k<j-i.$$
Hence, the $\rc$-code of the tower equals the $\rc$-name of length
$j-i$ of any point from the top of the tower.

Now, we introduce some notation which helps us to deal with
$\rc$-names of the points. For $n\in\NN$, a \emph{word} (or block) of
length $n$ over a finite set $\Sigma$ is any finite sequence
$u=u_0\ldots u_{n-1}$ of elements from $\Sigma$. The set of all words
of length $n$ is denoted by $\Sigma^n$, the length of $u$ is denoted
by $|u|$.
The set of all words of all lengths is denoted by $\Sigma^*$, i.e.,
$\Sigma^*=\bigcup_{n\in\NN}\Sigma^n$. For two words $u,v\in \Sigma^*$
we define their \emph{concatenation}, denoted simply by $uv$, as a word from
$\Sigma^{|u|+|v|}$ such that $(uv)_i=u_i$ if $i< |u|$, and
$(uv)_i=v_{i-|u|}$ if $|u|\leq i<|u|+|v|$. The concatenation of $k$
copies of a word $u$ is denoted by $u^k$. For $u\in \Sigma^*$,
$m,n\in\NN$, $m\leq n\leq |u|$, we denote by $u[m,n)$ the subword of
$u$ given by the interval $[m,n)$, i.e., $u[m,n)\in \Sigma^{n-m}$ and
$$
u[m,n)_i=u_{m+i},\qquad \text{for every } i \text{ such that } 0\leq i<m-n.
$$
The \emph{shift} mapping $S:\Sigma^*\to \Sigma^*$ is defined by
$S(u)=u[1,|u|)$ for all $u \in \Sigma^*$.

The shift mapping extends to infinite sequences in the following way.
Let $\Sigma^\NN$ be the product space of countably many copies of a
finite discrete space $\Sigma$.  The \emph{shift} mapping
$S:\Sigma^\NN\to\Sigma^\NN$ is defined by the following equality:
$(S(x))_i=x_{i+1}$ for $x = (x_i)_{i=0}^{+\infty} \in\Sigma^\NN$ and $i\in\NN$.
The mapping is continuous on $\Sigma^\NN$ and the pair
$(\Sigma^\NN,S)$ is a \emph{full shift}.  Given a closed $S$-invariant
subset $\Gamma\in\Sigma^\NN$ (i.e., $S(\Gamma)\subset \Gamma$), the
pair $(\Gamma,S)$ is a topological dynamical system called
\emph{subshift}, where $S$ is considered to be restricted to $\Gamma$.

A classical way to produce a subshift is the coding of an arbitrary
mapping $T: X\to X$ with respect to a finite partition of $X$.  Let
$\rc=\{R_i \mid i\in \Sigma\}$ denote a finite partition of $X$
labeled by $\Sigma$.  If $\Phi_\rc:X\to\Sigma^\NN$ is a mapping which
maps $x$ to its \emph{infinite $\rc$-name}, i.e.,
$\Phi_{\rc}(x)=(\phi_\rc(T^ix))^{+ \infty}_{i=0}$, then the mapping
$\Phi_\rc$ commutes with $T$ and $S$:
$$\Phi_\rc\circ T=S\circ\Phi_\rc.$$
In particular, the set $\Phi_\rc(X)$ and its closure are both
invariant under $S$.  Hence $(\overline{\Phi_\rc(X)},S )$ is a
subshift.

\subsection{Sturmian partition}
\label{sec:sturmian}

From now on, the mapping $T$ is as given in Introduction, i.e., it is
the rotation of the unit circle by an irrational angle
$\alpha$: $T(x)=(x+\alpha) \bmod{1}$ for all $x \in \RR/\ZZ$.  Also
recall that $\mathcal{P}$ is the partition consisting of two sets $P_0$ and $P_1$, where $P_{0}=[0,1-\alpha)$ and
$P_{1}=[1-\alpha,1)$.  Although both these objects depend on $\alpha$,
we do not explicitly state this dependence to ease the notation.

As already stated, the mapping $T$ and partition $\pc$ define a
Sturmian subshift $(\overline{\Phi_\pc([0,1))},S )$.  In this section,
we state some results concerning our interest: some specific
refinements of the partition $\pc$.

It is well-known that the partition $\pc^n$ consists of $n+1$
intervals. The Three lengths theorem (see \cite{So58}) says that these
intervals are of at most three lengths and specifies the lengths in
terms of convergents.  In particular, it is shown that the partition
$\pc^{r_k-1}$, for $k\in\NN$, has intervals of just two lengths. In
geometric proofs of the Three lengths theorem, not only the lengths of
intervals are determined, but also their endpoints.  A version of
the theorem, which is needed later, is recalled as Proposition
\ref{pro:twolength}. It is a special case of the description of
intervals from $\pc^n$ used in the proof of the Three lengths theorem
in \cite{Ku03-2} (Theorem 4.45, p. 160).

Before stating this version of the theorem, we need some more notations.
First, let us notice that the numbers $q_k\alpha-p_k$ for $k\in\NN$ form an
alternating sequence and their absolute values
$\eta_k=|q_k\alpha-p_k|$ satisfy the implicit formula
$\eta_k=\eta_{k-2}-c_k\eta_{k-1}$, $\eta_0=\alpha$, and
$\eta_1=1-c_1\alpha$.  Denote the following sequence of intervals in
$\RR/\ZZ$ as follows:
$$I_k = \begin{cases}
  [-\eta_k,0) & \mbox{for $k$ even,} \\ 
  [0, \eta_k ) & \mbox{for $k$ odd}.
  \end{cases}
$$

\begin{pro}\label{pro:twolength}
  Let $k\in\NN$. The partition $\pc^{r_k-1}$ consists of two Rokhlin
  towers $\Lambda(I_{k},0,q_{k-1})$ and $\Lambda(I_{k-1},0,q_k)$, i.e.
$$\pc^{r_k-1}=\Lambda(I_k,0,q_{k-1})\cup\Lambda(I_{k-1},0,q_k),\qquad \text{ for } k\geq 1.$$

Moreover,
$$T^{-(q_{k-1}+sq_k)}I_k\subset I_{k-1},\qquad \text{ for every } s \text{ such that } 0\leq
s< c_{k+1}.$$
\end{pro}

We also sometimes say that the towers $\Lambda(I_k,0,q_{k-1})$ and  $\Lambda(I_{k-1},0,q_k)$ \emph{form the partition} $\pc^{r_k-1}$.

We conclude this section with an iterative formula for $\rc$-codes of
the towers from the previous proposition.
\begin{lem}\label{lem:itercode}
  Let $k\in\NN$ and $\rc$ be a partition indexed by $\Sigma$. 
Let the $\rc$-codes of the towers $\Lambda(I_{k},0,q_{k-1})$ and $\Lambda(I_{k-1},0,q_k)$ exist
and let  $u\in\Sigma^{q_{k-1}}$ be the $\rc$-code of $\Lambda(I_{k},0,q_{k-1})$
 and $v\in\Sigma^{q_k}$ be the $\rc$-code of $\Lambda(I_{k-1},0,q_k)$.
Then $v$ and $v^{c_{k+1}}u$ are the $\rc$-codes of the towers $\Lambda(I_{k+1},0,q_k)$ and $\Lambda(I_{k},0,q_{k+1})$,
  respectively.
\end{lem}

The proof of the lemma is a straightforward application of Proposition
\ref{pro:twolength}.

The previous lemma and proposition are illustrated in Figures
\ref{fig:two-towers}, \ref{fig:labeling-the-towers}, \ref{fig:scheme}
and \ref{fig:two-tower} (for $k$ even, for an odd $k$ the higher tower is on the left in our setting).
The first three figures depict the two towers $\Lambda(I_{k},0,q_{k-1})$ and $\Lambda(I_{k-1},0,q_{k})$, the first being always on the left.
 Figure \ref{fig:two-towers} shows the
partition $\pc^{r_k-1}$ as two towers. Therein and in what follows, we
use a compact notation $\lra{n} := T^{-n}(0)$. In Figure
\ref{fig:labeling-the-towers}, the $\rc$-codes of the towers are
graphically presented. A level is labeled by the symbol $a\in\Sigma$
if and only if it is included in $R_a$.
The word $u = u_0 \cdots u_{q_{k-1}-1}$ is the $\rc$-code of $\Lambda(I_{k},0,q_{k-1})$
and $v = v_0 \cdots v_{{q_k}-1}$ of $\Lambda(I_{k-1},0,q_{k})$
Figure \ref{fig:scheme} shows the
dynamics given by $T$. The arrows indicate that each level, except the
base, is mapped to the level below. In accordance with the dynamics,
the $\rc$-codes are written in the towers from the top to the
base. The arrows also show where the points from the base
are mapped by $T$. By Proposition \ref{pro:twolength} the right part of
the base  of the right tower of length $\eta_k$ is mapped to the top of the left
tower. The rest of the base and the base of the left tower must be mapped to the top of the right
tower due to the injectivity of the mapping $T$. 
Figure \ref{fig:two-tower} illustrates two consecutive iterations of $\rc$-codes given by the last lemma. (The figure is for $k$ odd.)

  \begin{figure}[h]
    \centering

\begin{tikzpicture}[xscale=1,yscale=0.8]
  \pgfmathsetmacro{\s}{2} 
  \pgfmathsetmacro{\t}{7} 
  \pgfmathsetmacro{\r}{0.7} 
  \pgfmathsetmacro{\m}{0.1} 


  \draw[|-|] (0,0) node[below] {$\lra{q_k}$} -- node[below] {$I_k$}
  (\s,0) node[below] {$0$}; 

  \draw[|-|] (0,1) node[below] {$\lra{q_k+1}$} -- (\s,1) node[below]
  {$\lra{1}$};

  \draw[dotted] (\s/2,1.5) -- (\s/2,2.5);

  \draw[|-|] (0,3) node[below] {$\lra{r_k-1}$} -- (\s,3) node[below]
  {$\lra{q_{k-1}-1}$};


  \draw[|-|] (\s,0) -- node[below] {$I_{k-1}$} (\s+\t,0) node[below]
  {$\lra{q_{k-1}}$}; \draw[|-|] (\s,1) -- (\s+\t,1) node[below]
  {$\lra{q_{k-1}+1}$};

  \draw[dotted] (\s+\t/2,1.5) -- (\s+\t/2,2.5);

  \draw[|-|] (\s,3) -- (\s+\t,3) node[below] {$\lra{2q_{k-1}-1}$};


  \draw[|-|] (\s,4) node[below] {$\lra{q_{k-1}}$} -- (\s+\t,4)
  node[below] {$\lra{2q_{k-1}}$};

  \draw[dotted] (\s+\t/2,4.5) -- (\s+\t/2,8.5);

  \draw[|-|] (\s,9) node[below] {$\lra{q_{k}-1}$} -- (\s+\t,9)
  node[below] {$\lra{r_k-1}$};

\end{tikzpicture}
\caption{Rokhlin towers $\Lambda(I_k,0,q_{k-1})$ and
  $\Lambda(I_{k-1},0,q_{k})$ forming the partition
  $\pc^{r_k-1}$. (The figure is for $k$ even.)} 
\label{fig:two-towers}
\end{figure}
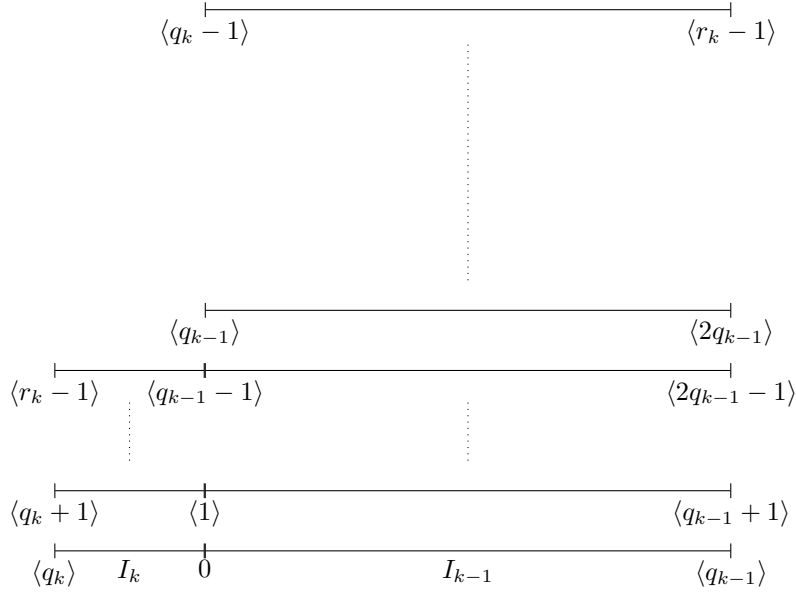


  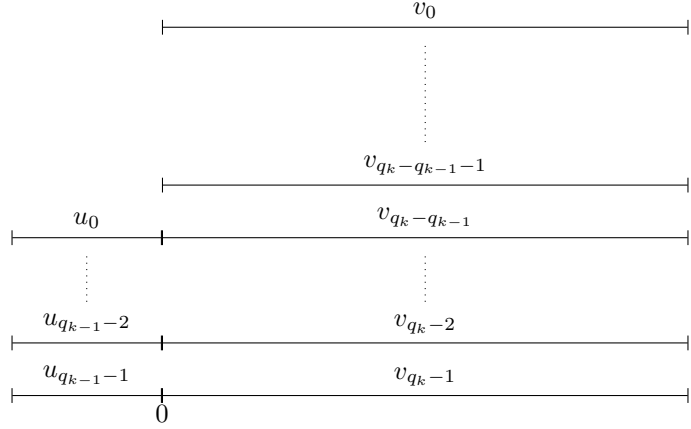
\begin{figure}[h]
    \centering
    \begin{tikzpicture}[xscale=1,yscale=0.7]
      \pgfmathsetmacro{\s}{2} 
      \pgfmathsetmacro{\t}{7} 

      \pgfmathsetmacro{\r}{0.7} 
      \pgfmathsetmacro{\m}{0.1} 


      \draw[|-|] (0,0) -- node[above] {$u_{q_{k-1}-1}$} (\s,0)
      node[below] {$0$}; \draw[|-|] (0,1) -- node[above]
      (endofdottedline1) {$u_{q_{k-1}-2}$} (\s,1); \draw[|-|] (0,3) --
      node[above] (u0) {$u_0$} node[below] (startofdottedline1) {}
      (\s,3); \draw[-,dotted] (startofdottedline1) to
      (endofdottedline1);


      \draw[|-|] (\s,0) -- node[above] {$v_{q_k-1}$} (\s+\t,0);
      \draw[|-|] (\s,1) -- node[above] (endofdottedline2)
      {$v_{q_k-2}$} (\s+\t,1); \draw[|-|] (\s,3) -- node[above]
      {$v_{q_k-q_{k-1}}$} node[below] (startofdottedline2) {}
      (\s+\t,3); \draw[-,dotted] (startofdottedline2) to
      (endofdottedline2);


      \draw[|-|] (\s,4) -- node[above] (endofdottedline3)
      {$v_{q_k-q_{k-1}-1}$} (\s+\t,4);

      \draw[|-|] (\s,7) -- node[above] {$v_0$} node[below]
      (startofdottedline3) {} (\s+\t,7);

      \draw[-,dotted] (startofdottedline3) to (endofdottedline3);

    \end{tikzpicture}
    \caption{$\rc$-codes of the towers forming $\pc^{r_k-1}$, word $u = u_0 \cdots u_{q_{k-1}-1}$ is the $\rc$-code of $\Lambda(I_k,0,q_{k-1})$, word $v = v_0 \cdots v_{{q_k}-1}$ is the $\rc$-code of $\Lambda(I_{k-1},0,q_{k})$. (The figure is for $k$ even.)}
    \label{fig:labeling-the-towers}
  \end{figure}

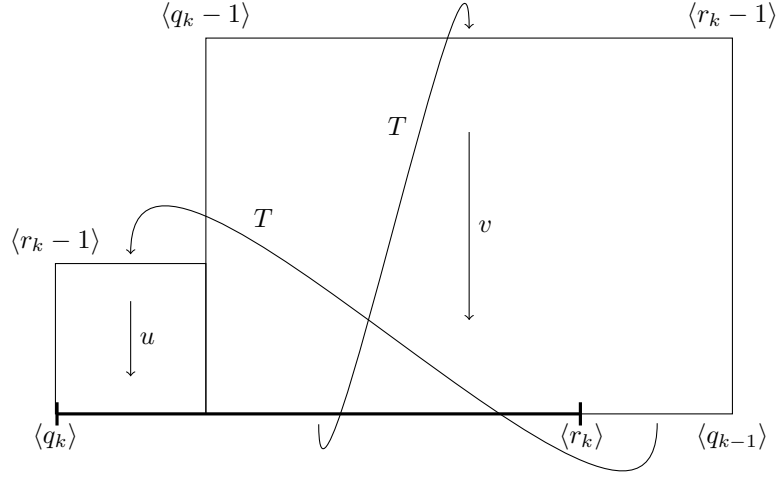
\begin{figure}
  \centering
  \begin{tikzpicture}

    \pgfmathsetmacro{\WL}{2} 
    \pgfmathsetmacro{\WR}{7} 
    \pgfmathsetmacro{\HL}{2} 
    \pgfmathsetmacro{\HR}{5} 

\node (recL) [draw,rectangle,minimum width=\WL cm,minimum
    height=\HL cm,outer sep=0pt] {}; 
\draw[->]
    ($(recL.north)!0.25!(recL.south)$) to node[right] {$u$}
    ($(recL.north)!0.75!(recL.south)$); \node (recR) at (recL.south
    east) [anchor=south west,draw,rectangle,minimum width=\WR
    cm,minimum height=\HR cm,outer sep=0pt] {}; 
\draw[->]
    ($(recR.north)!0.25!(recR.south)$) to node[right] {$v$}
     ($(recR.north)!0.75!(recR.south)$);

    \node at (recL.south west) [below] {$\lra{q_{k}}$};
    \node at (recL.north west) [above] {$\lra{r_{k}-1}$};
    \node at (recR.north west) [above] {$\lra{q_{k}-1}$}; \node at
    (recR.north east) [above] {$\lra{r_{k}-1}$}; \node at (recR.south
    east) [below] {$\lra{q_{k-1}}$};

    \coordinate (rk) at ($(recR.south east)-(\WL,0)$); \node at (rk)
    [anchor=north] {$\lra{r_k}$};

    \draw[|-|,very thick] (recL.south west) to node (startpoint) {}
    (rk); \node (endpoint) at (recR.north) {}; \draw
    [->,out=-90,in=90] (startpoint) to node[pos=0.6,above left] {$T$} (endpoint);

    \node (startpoint) at ($(recR.south east)!0.5!(rk)$) {}; \node
    (endpoint) at (recL.north) {}; \draw [->,out=-90,in=90]
    (startpoint) to node[pos=0.7,above right] {$T$} (endpoint);




    


  \end{tikzpicture}

  \caption{Dynamics (given by $T$) on the towers $\Lambda(I_k,0,q_{k-1})$ and  $\Lambda(I_{k-1},0,q_{k})$. (The figure is for $k$ even.)}
  \label{fig:scheme}
\end{figure}

\begin{figure}[h]
  \centering \subfloat[$\rc$-codes of the towers $\Lambda(I_{k},0,q_{k+1})$ and $\Lambda(I_{k+1},0,q_{k})$ forming $\pc^{r_{k+1}-1}$.] {
    \begin{tikzpicture}
      \node (recL1) [draw,rectangle,minimum width=3cm,minimum
      height=1cm,outer sep=0pt] {$\downarrow u$}; \node (recR1) at
      (recL1.south east) [anchor=south west,draw,rectangle,minimum
      width=2cm,minimum height=2.5cm,outer sep=0pt] {$\downarrow v$};
      \node (recL2) at (recL1.north east) [anchor=south
      east,draw,rectangle,minimum width=3cm,minimum height=5cm,outer
      sep=0pt] {$\downarrow v^{a_{k+1}}$};

    \end{tikzpicture}
    \label{fig:k+1 two-tower}
  } \hspace{2ex} \subfloat[$\rc$-codes of the two towers $\Lambda(I_{k+2},0,q_{k+1})$ and $\Lambda(I_{k+1},0,q_{k+2})$ forming $\pc^{r_{k+2}-1}$.] {
    \begin{tikzpicture}
      \node (recL1) [draw,rectangle,minimum width=2cm,minimum
      height=2cm,outer sep=0pt] {$\downarrow v^{a_{k+1}}u$}; \node
      (recR1) at (recL1.south east) [anchor=south
      west,draw,rectangle,minimum width=3cm,minimum height=0.5cm,outer
      sep=0pt] {$\downarrow v$}; \node (recL2) at (recR1.north east)
      [anchor=south east,draw,rectangle,minimum width=3cm,minimum
      height=5cm,outer sep=0pt] {$\downarrow
        (v^{a_{k+1}}u)^{a_{k+2}}$};

    \end{tikzpicture}

    \label{fig:k+2 two-tower}
  }
\caption{ }
\label{fig:two-tower}
\end{figure}


\subsection{Sturmian-measurable partitions}
\label{sec:sturm-measurable}
Let us recall that a partition of $[0,1)$ is Sturmian-measurable if
it is a finite partition whose elements are finite unions of
right-closed left-open intervals with endpoints from the set of
preimages of zero $T^{-i}(0)$, $i\in\NN$.  The class of all
Sturmian-measurable partitions is closed under taking preimages and joins. In
particular, for all $m\in\NN$ the partition $\rc^m$ is Sturmian-measurable whenever $\rc$ is
Sturmian-measurable.

For a partition $\rc$ we define the set $\partial\rc$ as the union of
the boundaries of the sets from $\rc$. In this definition we consider the
topology of $\RR/\ZZ$ represented by the fundamental domain $[0,1)$. The elements of $\partial\rc$ are
\emph{cutpoints} of the partition $\rc$.  Put
$$\cc(\rc)=\{i\in\ZZ \mid T^{-i}(0) \in\partial\rc\}.$$ 
The numbers from this set are called {\em cut-indices} of $\rc$. The
terminology follows the fact that for Sturmian-measurable partition
$\rc$, the set $\partial\rc$ is the smallest set such that the sets in $\rc$
can be described as a finite union of intervals whose endpoints
belong to $\partial\rc$. Hence, the partition $\rc$ cuts the circle
just at the points from $\partial\rc$.

The cutpoints and cut-indices interplay with the lattice
operations and the transformation $T$ in the following manner:
\begin{align*}
  \partial(\rc\vee\rc')&=\partial\rc \cup \partial\rc',&\partial(T^{-j}\rc)&=T^{-j}(\partial\rc),\\
  \cc(\rc\vee\rc')&=\cc(\rc) \cup
  \cc(\rc'),&\cc(T^{-j}\rc)&=j+\cc(\rc),
\end{align*}
for every integer $j$ and partitions $\rc$ and $\rc'$. In particular,
\begin{align*}
  \min(\cc(T^{-i}(\rc^j))&=i+\min(\cc(\rc)),\\
  \max(\cc(T^{-i}(\rc^j))&=i+j-1+\max(\cc(\rc)),
\end{align*}
for every $i\in\ZZ$ and $j\in\NN$.

Let us reformulate the assumption of Theorem \ref{thm:precise} using
cut-indices. Given a Sturmian-measurable partition $\rc$, integers $s,m\in\NN$,
the partition $T^{-s}\pc^m$ is finer than $\rc$ if and only if the
cut-indices of $\rc$ belong to the interval $[s,s+m]$. If $\rc$ is
non-trivial, the smallest interval containing all the cut-indices is
the interval $[\min(\cc(\rc)),\max(\cc(\rc))]$. Hence, among all pairs
$(s,m)$ such that $T^{-s}\pc^m$ is finer than $\rc$, there is a pair
$$(s',m')=(\min(\cc(\rc)),\max(\cc(\rc))-\min(\cc(\rc))),$$
which maximizes the first coordinate and minimizes the second
simultaneously. Hence the largest $\ell$ and the least $n$ introduced
in Theorem \ref{thm:precise}  are uniquely determined by the partition $\rc$. The number $\ell$ equals the minimal cut-index of $\rc$ and $n$ equals the difference between the maximal and the minimal cut-index of $\rc$.

\section{Proof of the main result}
\label{sec:proof}
In this section we prove Theorem \ref{thm:precise} by applying Proposition \ref{pro:rk-1} which is stated below.
In the second part we prove the proposition itself using the analysis of a periodic structure in $\rc$-codes. Let us notice, that the proposition is a special case of Theorem \ref{thm:precise}.

\begin{pro}\label{pro:rk-1}
  Let $\rc$ be rougher than $\pc^{r_k-1}$ for some $k\in\NN$.
If $0$
  and $r_k-1$ are cut-indices of $\rc$, then
$$\rc^{r_{k+3}+r_k-1}=\pc^{r_{k+3}+2r_k-3}.$$
\end{pro}

\begin{proof}[Proof of Theorem \ref{thm:precise}]
  Let $\ell,n\in\NN$ and a partition $\rc$ satisfy the assumptions of
  Theorem \ref{thm:precise}, i.e., $\ell=\min(\cc(\rc))$,
  $n=\max(\cc(\rc))-\min(\cc(\rc))$, and $k\in\NN$ such that
  $r_{k-1}\leq n<r_k$.  Denote $\rc'=T^{\ell}(\rc^{r_k-n})$. Thus,
  since $\rc$ is finer than $T^{\ell}\pc^m$, we get
$$\rc'<T^{\ell}\left(\left(T^{-\ell}\pc^n\right)^{r_k-n}\right)=T^{\ell}\left(T^{-\ell}\left(\pc^n\right)^{r_k-n}\right)=\pc^{r_k-1}.$$
In particular, $\rc'$ is Sturmian-measurable.

By the properties of cut-indices, mentioned in the previous section,
\begin{align*}
  \min(\cc(\rc'))&=-\ell+\min(\cc(\rc))=0,\\
  \max(\cc(\rc'))&=-\ell+r_k-n-1+\max(\cc(\rc))=-\ell+r_k-n-1+\ell+n=r_k-1.
\end{align*}
Thus, the partition $\rc'$ satisfies the assumptions of Proposition
\ref{pro:rk-1}. Applying the proposition we get
\begin{align*}
  T^{-\ell}\left(\pc^{r_{k+3}+2r_k-3}\right)&=T^{-\ell}\left((\rc')^{r_{k+3}+r_k-1}\right)= T^{-\ell}\left( \left(T^\ell\rc^{r_k-n}\right)^{r_{k+3}+r_k-1}\right)\\
  &=\rc^{r_{k+3}+2r_k-n-2}.
\end{align*}
 To complete the proof we need to prove Proposition \ref{pro:rk-1}.
\end{proof}

\subsection{Proof of Proposition \ref{pro:rk-1}}
\label{sec:proof-pro-rk-1}

In the rest of this section we fix $k\in\NN$ and a partition
$\rc=\{R_a,a\in\Sigma\}$ satisfying the assumptions of Proposition
\ref{pro:rk-1}, i.e., $k\geq 1$, $\rc$ is rougher than $\pc^{r_k-1}$
and the indices $0$ and $r_k-1$ are cut-indices of $\rc$.

Let us denote by $u$ and $v$ the $\rc$-codes of the towers
$\Lambda(I_k,0,q_{k-1})$ and $\Lambda(I_{k-1},0,q_{k})$ respectively
(see Figure \ref{fig:labeling-the-towers}).
The condition on the cut-indices of $\rc$ can be rephrased into the
following conditions on $u$ and $v$:
\begin{itemize}
\item since the cutpoint $0$ of $\rc$ is the common endpoint of the
  bases of the towers, the beginning of $u$ and $v$ differs,
  i.e., $u_0\neq v_0$;
\item since the cutpoint $T^{-(r_k-1)}0$ is the common endpoint of the
  tops of the towers, the end of $u$ and $v$ differs, i.e.,
  $u_{q_{k-1}-1}\neq v_{q_k-1}$.
\end{itemize}

Put
$$w=v^{c_{k+1}}u,\qquad  w'=(w)^{c_{k+2}}v,\qquad w''=(w')^{c_{k+3}}w,\qquad z=w''w'.$$ 
Applying Lemma \ref{lem:itercode} three times we get that $w$, $w'$
and $w''$ are the $\rc$-codes of the towers
$\Lambda(I_{k},0,q_{k+1})$, $\Lambda(I_{k+1},0,q_{k+2})$ and
$\Lambda(I_{k+2},0,q_{k+3})$, respectively. It implies that the words
$w$ and $w'$ are also the $\rc$-codes of the towers
$\Lambda(I_{k+2},0,q_{k+1})$ and $\Lambda(I_{k+3},0,q_{k+2})$. We have
just described the $\rc$-codes of the towers which form the partitions
$\pc^{r_{k+1}-1}$, $\pc^{r_{k+2}-1}$ and $\pc^{r_{k+3}-1}$. Our aim is
to use these $\rc$-codes to describe $\rc$-names of length $r_{k+3}$ of
the points from distinct sets from $\pc^{r_{k+3}+r_{k+1}-2}$. The key
role throughout this section will be played by the word $z$, which is
the $\rc$-code of the tower $\Lambda(I_{k+3},0,r_{k+3})$.

\begin{figure}[h]
  \begin{tikzpicture}


    \pgfmathsetmacro{\WL}{2} 
    \pgfmathsetmacro{\WR}{7} 
    \pgfmathsetmacro{\HL}{2} 
    \pgfmathsetmacro{\HB}{0.8} 
    \pgfmathsetmacro{\HR}{5} 

    \node (A) [draw,rectangle,minimum width=\WL cm,minimum height=\HL
    cm,outer sep=0pt] {$A$}; \node (leftm) at (A.west) [anchor=east]
    {$\downarrow w'$}; \node (recR) at (A.south east) [anchor=south
    west,draw,rectangle,minimum width=\WR cm,minimum height=\HR
    cm,outer sep=0pt,label=right:$w'' \downarrow$] {};
    \node (B) at (recR.south east) [anchor=south east, rectangle,
    minimum width=\WL cm, minimum height=\HB cm, outer sep=0pt] {$B$};
    \pgfmathsetmacro{\HC}{\HR-\HB} \node (C) at (recR.north east)
    [anchor=north east,rectangle,minimum width=\WL cm,minimum
    height=\HC cm,outer sep=0pt] {$C$}; \pgfmathsetmacro{\WD}{\WR-\WL}
    \node (D) at (recR.north west) [anchor=north
    west,rectangle,minimum width=\WD cm,minimum height=\HC cm,outer
    sep=0pt] {$D$}; \node (E) at (recR.south west) [anchor=south
    west,rectangle,minimum width=\WD cm,minimum height=\HB cm,outer
    sep=0pt
    ] {$E$};

    \draw[dashed] (C.south east) -- (C.south west) -- (C.north west);
    \draw [thick] (E.south east) to (E.north east); \draw (E.north
    east) -- (E.north west);

    \node at (A.south west) [below] {$\lra{q_{k+3}}$}; \node at
    (A.south east) [below] {$0$}; \node at (B.south west) [below]
    {$\lra{r_{k+3}}$}; \node at (B.south east) [below]
    {$\lra{q_{k+2}}$}; \node at (A.north west) [above]
    {$\lra{r_{k+3}-1}$}; \node at (D.north west) [above]
    {$\lra{q_{k+3}-1}$}; \node (lastcutpoint) at (B.north west) {};
    \node at (C.north east) [above] {$\lra{r_{k+3}-1}$};

    \node (lastcutpoint) at (B.north west) {}; \node
    (labeloflastcutpoint) at ($(B.north east)+(0.5,0.5)$) [right]
    {$\lra{r_{k+3}+r_k-2}$}; \draw [->,out=180,in=45]
    (labeloflastcutpoint) to (lastcutpoint);
  \end{tikzpicture}

  \caption{Important parts of $\pc^{r_{k+3}+r_k-2}$. (The figure is for $k$ even.)}
  \label{fig:rc_k+3}
\end{figure}

We divide the interval $[0, 1)$ into several Rokhlin towers (see Figure
\ref{fig:rc_k+3}) and separately analyze the $\rc$-names of length $r_{k+3}$ of
the points from distinct towers. 

Put
\begin{align*}
  A=&\Lambda(I,0,q_{k+2}),& B=&\Lambda(K,0,r_k-1),&C=&\Lambda(K,r_k-1,q_{k+3}),\\
  D=&\Lambda(J,r_k-1,q_{k+3}),&E=&\Lambda(J,0,r_k-1),&&
\end{align*}
where $I=I_{k+3}$, $J=I_{k+2} \setminus T^{-q_{k+2}}I_{k+3}$ and
$K=T^{-q_{k+2}}I_{k+3}$. The intervals $I$, $J$ and $K$ are the bases
of the towers $A$, $E$ and $B$, respectively. For any of the towers
$A, B, C, D$ and $E$, denote its union using the tilde over the
letter, e.g., $\widetilde A=\bigcup A$.

According to the Three lengths theorem the partition
$\pc^{r_{k+3}+r_k-2}$ arises from the partition $\pc^{r_{k+3}-1}$ by
adding new cutpoints
$$\lra{j},\quad \text{ for } j \text{ such that } r_{k+3}\leq j\leq r_{k+3}+r_k-2.$$

These points are illustrated by the vertical line between the towers
$B$ and $E$ in Figure \ref{fig:rc_k+3}. It implies that
$$\pc^{r_{k+3}+r_k-2}=A\cup B\cup E\cup\Lambda(I_{k+3},r_k-1,q_{k+3}).$$

For a point $x\in [0,1)$, denote the $\rc$-name of $x$ of length
$r_{k+3}$ by ${\widehat x}$. In addition, denote the addition and
subtraction in the finite modular group $\ZZ_{r_{k+3}}$ by $\oplus$
and $\ominus$. For a word $u\in\Sigma^{r_{k+3}}$, put
$$Per(u)=\{j\in\ZZ_{r_{k+3}} \mid u_j=u_{j\ominus |w'|} \},\qquad \sigma(u)=u_{r_{k+3}-1}u_0u_1 \ldots u_{r_{k+3}-2}u_0\in\Sigma^{r_{k+3}}.$$
Obviously, $Per(\sigma(u))=Per(u)\oplus 1$.

To find $Per(z)$ we need to compare $z$ and $\sigma^{|w'|}(z)$,
\begin{align*}
  z&= w''w'=(w')^{c_{k+3}}\ ww'=(w')^{c_{k+3}}\
  ww^{c_{k+2}}v=(w')^{c_{k+3}}\ w^{c_{k+2}}wv\\
  &=(w')^{c_{k+3}}\
  w^{c_{k+2}}v^{c_{k+1}}uv,\\
  \sigma^{|w'|}(z)&=w'w''=w'(w')^{c_{k+3}}w=(w')^{c_{k+3}}w'w=(w')^{c_{k+3}}\ w^{c_{k+2}}vw\\
  &=(w')^{c_{k+3}}\ w^{c_{k+2}}vv^{c_{k+1}}u=(w')^{c_{k+3}}\
  w^{c_{k+2}}v^{c_{k+1}}vu.
\end{align*}
One can see that the words coincide on first $|z|-|v|-|u|$
positions. Since the beginnings and ends of $u$ and $v$ differ, we get
that the words above differ at positions $|z|-1$ and
$|z|-|v|-|u|$. Thus,
$$ \NN \cap \left[0,|z|-|v|-|u|\right)\quad \subseteq\ \quad Per(z) $$
and
$$
  \{|z|-1, |z|-|v|-|u|\}\cap Per(z)=\emptyset.
$$

\begin{lem}\label{lem:ABC}
  If $x\in\widetilde A\cup \widetilde B\cup \widetilde C$, i.e., $x\in
  T^{-m}I$ for some $m\in [0,r_{k+3})$, then
$${\widehat x}=z[r_{k+3}-m-1,r_{k+3})z[0,r_{k+3}-m-1)=\sigma^{m+1-r_{k+3}}(z).$$
In particular,
$$Per({\widehat x})=Per(z)\oplus (m\oplus 1\ominus r_{k+3}).$$
\end{lem}

\begin{proof}
  First, let us prove that for $x\in T^{-(q_{k+3}-1)}I_{k+2}$, $\hat
  x=z$. Since $T^{-(q_{k+3}-1)}I_{k+2}$ is the top of the tower
  $\Lambda(I_{k+2},0,q_{k+3})$, we get that the beginning of $\hat x$
  equals the $\rc$-code of the tower, i.e.,
$$\hat x[0,q_{k+3})=w''=z[0,q_{k+3}).$$ 
Put $y=T^{q_{k+3}}x$.
Surely, $\hat x[q_{k+3},r_{k+3})$ equals the
$\rc$-name of $y$ of length $q_{k+2}$. Since $y\in T(I_{k+2})$,
where $I_{k+3}$ is the base of the tower
$\Lambda(I_{k+2},0,q_{k+3})$, the point $y$ should be on the top of
either of the towers $\Lambda(I_{k+3},0,q_{k+2})$ and
$\Lambda(I_{k+2},0,q_{k+3})$. In the former case, the $\rc$-name of
$y$ of length $q_{k+2}$ equals $w'$. In the latter case, the
$\rc$-name equals the beginning of $w''$ of length $q_{k+2}$,
i.e.,
$$\hat x[q_{k+3},r_{k+3})=w''[0,q_{k+2})=w'.$$
Altogether, $\hat x=w''w'=z$.

We proceed with the proof of the lemma. Let $x\in T^{-m}I$ for some
$0\leq m<q_{k+3}$. Put $y'=T^{m-(q_{k+3}-1)}x$, $y''=T^{m+1}x$. Thus,
$$y'\in T^{-(r_{k+3}-1)}I\subset T^{-(q_{k+3}-1)}I_{k+2},\qquad y''\in T(I)\subset T^{-(q_{k+3}-1)}I_{k+2}.$$
(For the last inclusion see the discussion below Proposition
\ref{pro:twolength}). We get that $\widehat{y'}$ and $\widehat{y''}$
equals $z$. As follows immediately from the definition,
\begin{align*}
  {\widehat x}[0,m+1)&=\widehat{y'}[r_{k+3}-m-1,r_{k+3})=z[r_{k+3}-m-1,r_{k+3}),\\
  {\widehat
    x}[m+1,r_{k+3})&=\widehat{y''}[0,r_{k+3}-m-1)=[0,r_{k+3}-m-1),
\end{align*}
which concludes the proof.
\end{proof}

\begin{lem}\label{lem:E}
  If $x$ is from $\widetilde E$, i.e., $x\in T^{-m}J$, $0\leq
  m<r_k-1$, then
$${\widehat x}=w''[q_{k+3}-m-1,q_{k+3})w''w'[0,q_{k+2}-m-1)$$
and neither $m\oplus |w'|$ nor $m\ominus r_k\oplus 1$ belong to
$Per(\widehat x)$.
\end{lem}

\begin{proof}
  Let $x$ be from $\widetilde E$, i.e., $x\in T^{-m}J$, $0\leq
  m<r_k-1$. Then $T^{m-(r_{k+3}-1)}x$ belongs to the top of the tower
  $\Lambda(I_{k+2},0,q_{k+3})$ whose $\rc$-code is $w''$. Hence, $\hat
  x[0,m+1)$ equals $w''[q_{k+3}-m-1,q_{k+3})$. Since the point
  $y=T^{m+1}x$ belongs to the top of the tower
  $\Lambda(I_{k+2},0,q_{k+3})$, $\hat y$ equals $z$ (see the proof
  of the previous proposition). It implies that
  \begin{align*}
    \hat x[m+1,r_{k+3})&=\hat y[0,r_{k+3}-m-1)=z[0,r_{k+3}-m-1)\\
    &=w''\, w'[0,q_{k+2}-m-1,q_{k+3}).
  \end{align*}
  We proved the first part of the lemma.

  The equality above implies that ${\widehat x}_{m}$ is the last
  letter of $w''$. But $w''$ ends with $w$ and $w$ ends with $u$. So,
  $x_m$ is equal to the last letter of $u$. Moreover,
$${\widehat x}[m+1,m+|w'|+1)=w''[0,|w'|)=w'=w^{c_{k+2}}v.$$
Thus, ${\widehat x}_{m+|w'|}$ equals the last letter of $v$ and so
differs from $x_m$, i.e., $(m\oplus |w'|)$ does not belong to $Per({\widehat x})$.

Since $m<r_k-1$ and $r_k<|w''|=q_{k+2}$, we get
$${\widehat x}_{m\ominus r_k\oplus 1}={\widehat x}_{m+1+|w''|+|w'|-r_k}=w'_{|w'|-r_k}$$
and
$${\widehat x}_{m\ominus r_k\oplus 1\ominus |w'| }={\widehat x}_{m+1+|w''|-r_k}=w''_{|w''|-r_k}.$$
However,
$$w'=w^{c_{k+2}}v=\overbrace{w^{c_{k+2}-1}v^{c_{k+1}}}^{|w'|-r_k}uv \quad \text{ and } \quad w''=(w')^{c_{k+3}}w=\overbrace{(w')^{c_{k+3}}v^{c_{k+1}-1}}^{|w''|-r_k}vu.$$

Since $r_k=|u|+|v|$, ${\widehat x}_{m\ominus r_k\oplus 1}$ equals the
first letter of $u$ and ${\widehat x}_{m\ominus r_k\oplus 1\ominus
  |w'| }$ equals the first letter of $v$. By the properties of $u$ and
$v$, the letters differ, i.e., $(m\ominus r_k\oplus 1)$ does not belong to $Per({\widehat
  x})$.
\end{proof}

\begin{lem}
  \mbox{}
  \begin{itemize}
  \item If $x\in T^{-m}I$ and $y\in T^{-m'}I$ for some $0\leq
    m<m'<r_{k+3}$, then ${\widehat x}\neq{\widehat y}$.
  \item If $x\in \widetilde{A}\cup\widetilde{B}\cup\widetilde{C}$ and
    $y\in \widetilde{E}$, then ${\widehat x}\neq{\widehat y}$.
  \end{itemize}
\end{lem}

 \begin{proof}
   Let $x\in T^{-m}I$ and $y\in T^{-m'}I$ for some $0\leq
   m<m'<r_{k+3}$. Since $r_{k+3}-r_k$ is greater than $r_k/2$, then
$$0<m\ominus m'<r_{k+3}-r_k\qquad \text{ or }\qquad 0<m'\ominus m<r_{k+3}-r_k.$$
Suppose that the former inequality holds. Then the number
$$m\ominus r_{k+3}=m'\oplus (m\ominus m') \ominus r_{k+3}$$
does not belong to $Per({\widehat x})$, but belongs to $Per({\widehat
  y})$ (see Lemma \ref{lem:ABC}). It implies that ${\widehat
  x}\neq{\widehat y}$.  By similar arguments, the latter of the above
mentioned inequalities implies that $m'\ominus r_{k+3}$ does not
belong to $Per({\widehat y})$, but belongs to $Per({\widehat x})$. We
get again that ${\widehat x}$ and ${\widehat y}$ differ.

Let the group $\ZZ_{r_{k+3}}$ be equipped with the ``circle'' distance
$d(i,j)$ defined as the minimum of $i\ominus j$ and $j\ominus i$. Let
$x\in \widetilde{A}\cup\widetilde{B}\cup\widetilde{C}$ and $y\in
\widetilde{E}$, i.e., $y\in T^{-m}J$ for some $m<r_k$.  By Lemma
\ref{lem:ABC} we get that
$$
\diam\left(\ZZ_{r_{k+3}}\setminus Per({\widehat x})\right) \leq j\leq
r_k,
$$
and by Lemma \ref{lem:E} we deduce that
$$
\diam \left(\ZZ_{r_{k+3}}\setminus Per(\widehat y)\right) \geq
d(m\oplus |w'|,m\ominus r_k\oplus 1)=\min(r_k\ominus 1\oplus
|w'|,1\ominus r_k\ominus |w'|).
$$
Since $r_k+|w'|\leq r_{k+3}$ and $|w'|\geq 2$, we get
\begin{align*}
  r_k\ominus 1\oplus |w'|&=r_k-1+|w'|>r_k \quad \text{ and } \\
  1\ominus r_k\ominus |w'|&=r_{k+3}+1-r_k-|w'|>q_{k+3}+q_{k+2}-q_k-q_{k-1}-q_{k+2}\\
  &=q_{k+2}+q_{k+1}-q_k-q_{k-1} \geq q_{k+2}\geq r_{k+1}>r_k.
\end{align*}
Thus, the diameters of the above mentioned sets differ. It implies
${\widehat x}\neq{\widehat y}$.
\end{proof}

\begin{figure}
  \begin{tikzpicture}
    \pgfmathsetmacro{\WL}{0.7} 
    \pgfmathsetmacro{\WR}{2.5} 
    \pgfmathsetmacro{\HL}{1.4} 
    \pgfmathsetmacro{\HB}{0.5} 
    \pgfmathsetmacro{\HR}{2} \pgfmathsetmacro{\WD}{\WR-\WL};
    \pgfmathsetmacro{\gapfirst}{1.1}
    \pgfmathsetmacro{\gapsecond}{10-3*\WR}
    \pgfmathsetmacro{\shiftXfirst}{(\gapfirst+\WL+\WR)}
    \pgfmathsetmacro{\shiftXsecond}{(\shiftXfirst+\gapsecond+\WL+\WR)}
    \node (A) at (0,0) [anchor=south west,draw,rectangle,minimum
    width=\WL cm,minimum height=\HL cm,outer
    sep=0pt,pattern=horizontal lines]{}; 
    \node (recR) at (A.south east) [anchor=south
    west,draw,rectangle,minimum width=\WR cm,minimum height=\HR
    cm,outer sep=0pt,,pattern=horizontal lines] {}; \node (E) at
    (recR.south west) [fill=white,draw,anchor=south
    west,rectangle,minimum width=\WD cm,minimum height=\HB cm,outer
    sep=0pt] {}; \coordinate (labl) at ($(recR.north)+(-\WL/2,2ex)$);
    \node at (labl) [anchor=south] {$\rc'$}; \node at (E.center)
    {$\widetilde E$};

    \node (gap) at (recR.south east) [anchor=south
    west,rectangle,minimum width=\gapfirst cm,minimum height=\HR
    cm,outer sep=0pt] {$\bigvee$};

    \node (A) at (gap.south east) [anchor=south west,rectangle,minimum
    width=\WL cm,minimum height=\HL cm,outer
    sep=0pt,pattern=horizontal lines]{}; 
    \node (recR) at (A.south east) [anchor=south
    west,rectangle,minimum width=\WR cm,minimum height=\HR cm,outer
    sep=0pt,,pattern=horizontal lines] {}; \node (E) at ($(recR.south
    west)+(0,\HB)$) [fill=white,draw,anchor=south
    west,rectangle,minimum width=\WD cm,minimum height=\HB cm,outer
    sep=0pt] {}; \coordinate (labl) at ($(recR.north)+(-\WL/2,2ex)$);
    \draw (A.north east) to (A.north west) to (A.south west) to
    (recR.south east) to (recR.north east) to (recR.north west) to
    (E.north west); \draw (E.south east) to ($(recR.south
    west)+(\WD,0)$); \node at (labl) [anchor=south]
    {$T^{-(r_k-1)}\rc'$}; \node at (E.center) {$T^{-(r_k-1)}\widetilde
      E$};

    \node (gap) at (recR.south east) [anchor=south
    west,rectangle,minimum width=\gapsecond cm,minimum height=\HR
    cm,outer sep=0pt] {}; \node at (gap.center) [scale=2] {$=$};

    \node (A) at (gap.south east) [anchor=south
    west,draw,rectangle,minimum width=\WL cm,minimum height=\HL
    cm,outer sep=0pt,pattern=horizontal lines]{}; 
    \node (recR) at (A.south east) [anchor=south
    west,draw,rectangle,minimum width=\WR cm,minimum height=\HR
    cm,outer sep=0pt,,pattern=horizontal lines] {}; \node (E) at
    (recR.south west) [draw,anchor=south west,rectangle,minimum
    width=\WD cm,minimum height=\HB cm,outer sep=0pt] {}; \node
    (Eprime) at (E.north west) [draw,anchor=south
    west,rectangle,minimum width=\WD cm,minimum height=\HB cm,outer
    sep=0pt] {};


    \coordinate (labl) at ($(recR.north)+(-\WL/2,2ex)$); \node at
    (labl) [anchor=south] {$\pc^{r_{k+3}+2r_k-3}$};

  \end{tikzpicture}

  \caption{Partitions $\rc'$, $T^{-(r_k-1)}\rc'$ and
    $\pc^{r_{k+3}+2r_k-2}$.}
  \label{fig:rcprime}
\end{figure}

  \begin{cor}
    Partition $\rc^{r_{k+3}}$ is finer than $\rc'$, where
$$\rc'=\{\widetilde E\}\cup A\cup B\cup \Lambda(J\cup K,r_k-1,q_{k+3}).$$
\end{cor}
\begin{proof}
  The previous lemma shows that
$$\rc^{r_{k+3}}|\widetilde X\setminus \widetilde D>\{\widetilde E\}\cup A\cup B\cup C.$$
Since $\rc^{r_{k+3}}<\pc^{r_{k+3}+r_k-2}$, the points in $\widetilde
D$ have the same $\rc$-names of length $r_{k+3}$ as the points from
the same level in $\widetilde C$. More precisely, if $x\in T^{-j}J$,
$r_k-1\leq j<q_{k+3}$, then $\widehat x=\widehat y$ for any $y\in
T^{-m}I_{k+2}$, in particular, $\widehat x=\widehat y$ for any $y\in
T^{-j}K$.
\end{proof}

  \begin{pro}
    $\rc^{r_{k+3}+r_k-1}=\pc^{r_{k+3}+2r_k-3}$.
  \end{pro}
  \begin{proof}
    Since $\rc<\pc^{r_k-1}$,
    $$\rc^{r_{k+3}+r_k-1}<\left(\pc^{r_k-1}\right)^{r_{k+3}+r_k}=\pc^{r_{k+3}+2r_k-3}.$$
    The opposite inequality arises as follows:
    \begin{align*}
\rc^{r_{k+3}+r_k-1}&\quad=\quad\left(\rc^{r_{k+3}}\right)^{r_k}\quad >\quad \rc^{r_{k+3}}\vee T^{-(r_k-1)}\rc^{r_{k+3}}\\
&\quad>\quad\rc'\vee T^{-(r_k-1)}\rc'\quad>\quad \pc^{r_{k+3}+2r_k-3}.
\end{align*}
The first inequality is obvious, the second holds by the previous
lemma. The last inequality follows from the fact, that
$T^{-(r_k-1)}\widetilde E$ is a subset of the union $(\widetilde
D\cup\widetilde C)$, where partition $\rc'$ separates each level, see
Figure \ref{fig:rcprime}.
\end{proof}

\begin{rmk}\label{rmk:Rauzy_Bratelli}
  We use the concept of refining Rokhlin towers over the intervals $I_n$ for $n=k-1,k,\ldots,k+3$ to find $\rc$-names of suitable lengths for the points from the unit interval.
This approach is based on the same ideas that are behind two classical concepts: Rauzy induction and Bratelli-Vershik diagrams.
Both notions were originally introduced as a multiscale description of ``low dimensional'' dynamics (not only the rotation), nevertheless, they can be adapted to provide a method for generating the coding of the trajectories with respect to a certain partition.
Rauzy induction for rotations is used e.g. in \cite{AFH99}, Bratelli-Vershik diagrams for rotations is described in \cite{DDM00}.
For reader's convenience, let us briefly explain how is our approach related to Rauzy induction.

Rauzy induction is based on the induction on a certain set and localization of points by their times of the first entrance to this set.
In this article, we in fact perform a Rauzy induction by inducing the original map $T$ on the set $J = I_n \cup I_{n+1}$ for $n \geq k$, where $k$ is the integer fixed in the beginning of section~\ref{sec:proof-pro-rk-1}.
By using the concept of Rokhlin towers, we do not recover a new two-interval exchange transformation as in Rauzy induction, but we get a pair of towers where the dynamics on their basis corresponds to the dynamics obtained by Rauzy induction.
The main difference of the two approaches is that the described Rauzy induction (and its recursive repeating) allows us to easily determine $\rc$-names of the points of the basis of the new towers, i.e., of interval $J$, while our approach gives easily the $\rc$-names of the points from the tops of the new towers, i.e., from the set $T^{-q_{n+1}+1}(I_n) \cup T^{-q_n+1}(I_{n+1})$ (see lemma~\ref{lem:itercode}).

\end{rmk}


\section{Symbolic Dynamics}
\label{sec:subshifts}
In this section, we rephrase our main results in the terms of Sturmian
subshifts and related sliding block codes.
We use the fact that a Sturmian subshift derived from the
rotation by an angle $\alpha$ arises as coding of the rotation with
respect to the partition $\pc$. We also show that in other subshifts
the analogous proposition need not hold.

\subsection{Sturmian subshifts}
\label{sec:sturmian-subshifts}
A \emph{Sturmian subshift} $(\Gamma,S)$ is the coding of the rotation
with respect to the partition $\pc$, i.e.,
$$\Gamma=\overline{\Phi_{\pc}(X)}\subset\{0,1\}^\NN.$$ 

The topology on $\Gamma$ is generated by the sets
$$[u]_\ell=\{(x_i)_{i\in\NN}\in \Gamma \mid x_{i+\ell}=u_i\text{ for every } 0\leq i<|u| \},\qquad u\in\{0,1\}^*.$$
The set $[u]_\ell$, if it is nonempty, is called a \emph{cylinder} of
length $n$ shifted by $\ell$.
The partition of cylinders of length $n$ shifted by $\ell$ is defined
as follows:
$$\left[\Sigma^n\right]_\ell=\{[u]_\ell \mid  u\in \Sigma^n\}\setminus\{\emptyset\}
.$$

The inverse mapping $\Phi^{-1}_\pc$, applied as a set function on the
subsets of $\Gamma$, has the following properties:
\begin{enumerate}
\item For every $\ell,n\in\NN$, the mapping sends $[\Sigma^n]_\ell$
  bijectively onto $T^{-\ell}\pc^n$, i.e.,
$$
\Phi^{-1}_\pc \left ( [u]_\ell \right ) = T^{-\ell} P_u
=\bigcap^{n-1+\ell}_{i= \ell}T^{-i}P_{u_i},\quad \text{ for every } u
\in\Sigma^n.
$$   
\item The mapping preserves the relation ``to be rougher than'', i.e.,
  if $\rc<\rc'$ for partitions of $\Sigma$, then
  $\Phi^{-1}_\pc(\rc)<\Phi^{-1}_\pc (\rc')$.
\end{enumerate}

It follows immediately that the following results hold.
\begin{pro}\label{pro:main-sturm}
  If $n\in\NN$ and $\rc$ is a nontrivial partition rougher than
  $\left[\Sigma^n\right]_0$, then there exist $k,\ell,m\in\NN$ such
  that $\ell<n$ and
$$ \rc^k=\left[\Sigma^n\right]_\ell.$$
\end{pro}

\begin{pro}\label{pro:precise-sturm}
  Let $\rc$ be a nontrivial partition rougher than
  $\left[\Sigma^n\right]_\ell$ for some $\ell,n\in\NN$. Take $\ell$
  the largest and $n$ the least possible to satisfy the assumption. If
  $k \in\NN$ such that $r_{k-1}\leq n<r_k$, then
$$\rc^{r_{k+3}+2r_k-n-2}=\left[\Sigma^{r_{k+3}+2r_k-3}\right]_\ell.$$
\end{pro}

These results can be also reformulated in terms of sliding block
codes.

\subsection{Sliding block codes}
\label{sec:block}

In this section we mainly follow the terminology from \cite{Ku03-2}
and \cite{LM95}. Given a subshift $(\Gamma,S)$ and $m \in \NN$, the
\emph{language of $\Gamma$ of length $m$} is the set of words defined
as follows:
$$\lc^m(\Gamma)=\{u[k,k+m) \mid u\in\Gamma, k\in\NN\}.$$

For positive integers $m$ and $n$ and a mapping $\psi$ from
$\lc^m(\Gamma)$ to a finite set $\Delta$, we denote by $\psi^{*n}$ the
mapping from $\lc^{m+n-1}(\Gamma)$ to $\Delta^n$ defined by the
equality
$$(\psi^{*n}(u))_i=\psi(u_iu_{i+1}\ldots u_{i+m-1}),\qquad 0\leq i<n, u\in\lc^{m+n-1}(\Gamma).$$
The mapping $\psi$ is called a \emph{local rule of width $m$} and
$\psi^{*n}$ is called the \emph{sliding block code of length $n$
  induced by $\psi$}. In the same way, the mapping from $\Gamma$ to
$\Delta^\NN$ defined by the
equality $$(\psi^{*\infty}(u))_i=\psi(u_iu_{i+1}\ldots
u_{i+m-1}),\qquad 0\leq i<n, u\in\Gamma,$$ is the \emph{infinite
  sliding block code induced by $\psi$}.

A \emph{homomorphism} from a subshift $(\Gamma,S)$ to a shift
$(\Delta^\NN,S)$ is any continuous mapping $f:\Gamma\to\Delta^\NN$
that commutes with shift mappings, i.e., $f\circ S=S\circ f$. Since
the spaces $\Gamma$ and $\Delta^\NN$ are compact, every homomorphism
$f$ is uniformly continuous and it is therefore equal to the infinite
sliding block code $\psi^{*\infty}$ for some local rule $\psi$.

The main problem of this section is the relation of the injectivity of
a local rule and the injectivity of the induced homomorphism.  The
following lemma shows that one direction follows from the definitions.

\begin{lem}\label{lem:injective_infinite}
  Let $(\Gamma,S)$ and $(\Delta^\NN,S)$ be subshifts,
  $\psi:\lc^m(\Gamma)\to\Delta$ be a local rule.  If for some
  $n\in\NN$ the sliding block code $\psi^{*n}$ is injective, then the
  sliding block code $\psi^{*(n+\ell)}$ is injective for every $\ell
  \in\NN$ and the infinite sliding block code $\psi^{*\infty}$ is
  injective.
\end{lem}

\begin{proof}
  Let $\psi:\lc^m(\Gamma)\to\Delta$ be a local rule and $n$ be a
  natural number such that $\psi^{*n}$ is injective.

  Given $\ell \in\NN$, suppose that words $u$ and $v$ from
  $\lc^{m+n+\ell-1}(\Gamma)$ have the same image under
  $\psi^{*(n+\ell)}$ which we denote by $w$. Then for every $i\leq
  \ell$,
$$\psi^{*n}(u[i,i+m+n-1))=w[i,i+n)=\psi^{*n}(v[i,i+m+n-1)).$$
Assuming injectivity of $\psi^{*n}$, we get that the words
$u[i,i+m+n-1)$ and $v[i,i+m+n-1)$ are the same for every $i\leq
\ell$. It implies that $u=v$. This proves that $\psi^{*(n+l)}$
is also injective.

The proof of injectivity of $\psi^{*\infty}$ is analogous.
\end{proof}

Let us emphasize the part of the lemma which says that if a sliding
block code of some finite length is injective, then the infinite sliding block
code is injective too. A natural question is
whether the converse holds.

If the subshift $(\Gamma,S)$ is finite, i.e., $\Gamma$ is finite, then
the situation is simple and the answer is affirmative. In the infinite
case, an important role is played by the minimality and the dependence
of the local rule on the first coordinate.

For every positive natural number $n\geq 2$, denote by $g_n$ the
mapping from $\lc^{n}(\Gamma)$ to $\lc^{n-1}(\Gamma)$ defined as the
\emph{cut-off of the last letter}, i.e., $g(x)=x[0,n-1)$. We say that
a local rule $\psi$ from $\lc^m(\Gamma)$ to $\Delta$ is \emph{minimal}
if either $m=1$, or there is no local rule $\psi'$ from
$\lc^{m-1}(\Gamma)$ to $\Delta$ satisfying the condition:
$\psi'=\psi\circ g_m$, i.e., $\psi'(u)=\psi(u[0,m-1))$ for every
$u\in\lc^m(\Gamma)$.  A local rule
$\psi:\lc^m(\Gamma)\to\lc^1(\Gamma)$ \emph{ignores the first letter}
if there exists a mapping $\psi':\lc^{m-1}(\Gamma)\to\lc^1(\Gamma)$
such that $\psi=\psi'\circ S$, i.e., $\psi(u)$ equals $\psi'(u[1,m))$,
for every $u\in\lc^m(\Gamma)$.

A rule ignoring the first letter induces sliding block codes which
ignores the first letter as well. In other words, if $\psi=\psi'\circ
S$, then for every $n\in\NN$ we have $\psi^{*n}=(\psi')^{*n}\circ S$
and $\psi^{*\infty}=(\psi')^{*\infty}\circ S$. Since the mapping $S$
is not injective, the sliding block codes of all lengths and the infinite sliding block code are not injective either.

The next two lemmas show the important role of minimal local rules.
\begin{lem}
  If $\psi$ is a local rule of width $m$, then a local rule $\psi'$
  of width $m'$ such that $m'\leq m$ induces the same homomorphism as $\psi$ if
  and only if
$$\psi'(u[0,m'))=\psi(u), \text{ for every $u\in\lc^m(\Gamma)$}.$$

In particular, given a local rule $\psi$, there exists just one
minimal local rule $\psi'$ of a smaller or equal width that induces
the same homomorphism. The mapping $\psi'$ is of minimal width among
all mappings inducing the same homomorphism as $\psi$.
\end{lem}
The proof is straightforward.

\begin{lem}
  Let $\Gamma$ be infinite, $\psi:\lc^m(\Gamma)\to\Delta$ be a local
  rule of width $m$.  If a sliding block code of finite length induced by $\psi$
  is injective, then $\psi$ is minimal.
\end{lem}
\begin{proof}
  Suppose that $\Gamma$ is infinite, $\psi$ is a non-minimal local
  rule of width $m$, $n$ is a positive natural number.  Since $\psi$
  is not minimal, there exists a local rule $\psi'$ of width $m-1$
  such that $\psi=\psi'\circ g_m$. It is readily seen that for every
  natural number $n\in\NN$ we have $\psi^{*n}=(\psi')^{*n}\circ
  g_{m+n-1}$. But the infiniteness of $\Gamma$ does not allow
  $g_{m+n-1}$ to be injective. Hence, $\psi^{*n}$ is not injective
  either.
\end{proof}

In accordance with the previous discussion, we can restrict the above mentioned question as follows:

\begin{qes} \label{qes:1}
  Let a minimal local rule $\psi$ do not ignore the first letter. Does
  the injectivity of the infinite  sliding block code induced by $\psi$ implies
  the injectivity of the sliding block code of length $n$ induced by $\psi$ for some
  finite $n\in\NN$?
\end{qes}

The answer to this question is negative: the following two examples, Example \ref{exm:symbolic} and \ref{exm:symbolic_minimal}, are counterexamples.
However, we will show that if we restrict ourselves to Sturmian subshifts, then the answer is positive (see Proposition \ref{pro:sturm-sliding} below).

\begin{exm}\label{exm:symbolic}
  Let $\Gamma=\{0,1\}^\NN$ and $\Delta=\{0,1,2\}$. Let
  $\psi:\lc^2(\Gamma)\to\Delta$ be the local rule defined as follows:
$$\psi(11)=0, \ \psi(10)=0, \ \psi(01)=1, \ \psi(00)=2. $$
This rule is minimal and does not ignore the first letter. Let us
remark that for every $x,y\in\lc^2(\Gamma)$, $\psi(x)=\psi(y)$ implies
$x[0,1)=y[0,1)$. By induction one can easily prove that for every
$n\in\NN$,
$$\psi^{*n}(x)=\psi^{*n}(y)\qquad\Longrightarrow\qquad x[0,n)=y[0,n).$$
Hence, $\psi^{*\infty}(x)=\psi^{*\infty}(y)$ implies $x=y$. We get
that the infinite sliding block code induced by $\psi$
is injective. On the other hand, for every $n\in\NN$, $x\in\{0,1\}^n$,
the words $x10$ and $x11$ belong to $\lc^{n+2}(\Gamma)$ and their
images under $\psi^{*(n+1)}$ are the same. It implies that no sliding
block code of finite length induced by $\psi$ is injective.
\end{exm}

The previous example is simple and instructive, however, it is quite far from Sturmian subshift, for which the answer is positive as already mentioned, from the point of view of subword complexity and minimality.
Therefore, we introduce the next example of a subshift which is minimal and of low subword complexity.

\begin{exm}\label{exm:symbolic_minimal}
  Let $\Gamma$ be the Toeplitz subshift generated by a Toeplitz sequence given as the limit of the following sequence of words $(u_n)$ determined by
$$u_1=00 \quad \quad \text{ and } \quad \quad u_{n+1}=u_n11u_n10u_n \quad \text{ for } n > 0.$$

Surely, all words $11$, $10$, $01$ and $00$ belong to the language of $\Gamma$. Hence, the local rule $\psi$ has the same domain as in Example \ref{exm:symbolic}. Therefore, the local rule is minimal and does not ignore the first letter. In addition, the words $u_n11$ and $u_n10$ belong to the language of $\Gamma$ and have the same image under the sliding block code of length $n+1$ induced by $\psi$. Thus, the sliding block code of finite length induced by $\psi$ is never injective.
\end{exm}

Let us remark that a Toeplitz sequence is of linear subword complexity, see \cite{Ca97}.
Hence, the examples show that injectivity of an infinite sliding block code (i.e. the factor mapping from a subshift) does not imply the existence of an injective sliding block code of finite length.

Before introducing the positive answer to Question \ref{qes:1} in the case of Sturmian subshifts we would like to discuss a closely related fact about Cantor primarility of the two-sided Sturmian subshift. The proof of this fact introduced in \cite{Du00} shows slightly more, namely that any non-trivial factor mapping from a two-sided Sturmian subshift to a dynamical system defined on a Cantor space is injective (so the map is a conjugacy). The proof can be adapted for one-sided Sturmian subshifts, but then a new technical assumption on the mapping appears. If we focus on the factor mappings to a subshift, the claim and its proof is as follows.

\begin{pro} \label{pro:Durand}
Let $\psi^{*\infty}$ be a homomorphism from a (one-sided) Sturmian subshift $(\Gamma,S)$ to a (one-sided) subshift $(\Delta^\NN,S)$ induced by a minimal local rule $\psi:\lc^m(\Gamma) \to\Delta$.
If $\psi$ does not ignore the first letter, then the two following conditions are equivalent:
\begin{itemize}
\item $\psi^{*\infty}$ is not constant,
\item $\psi^{*\infty}$ is injective.
\end{itemize}
\end{pro}

\begin{proof}
Since injectivity implies non-triviality, we only need to prove that there is no minimal rule $\psi$ such that $\psi$ does not ignore the first letter and $\psi^{*\infty}$ is neither constant, nor injective. Let us assume such a rule $\psi$ exists. We will now follow the ideas in the proof of Proposition 11 in \cite{Du00}. For a (one-sided) Sturmian subshift $(\Gamma,S)$ there is a factor mapping $\gamma$ from $(\Gamma,S)$ to the rotation of the circle $([0,1[,T)$, where $T(z)=(z+\alpha) \bmod{1}$, such that every point from the circle has a unique preimage except for the points $T^{-n}(0)$ with $n\in\NN$. Denote the set of these points by $\mathcal O$. The point $T^{-n}(0)$ has two preimages which we denote $x(n)$ and $y(n)$ in such a way that the conditions $S(x(n+1))=x(n)$ and $S(y(n+1))=y(n)$ hold for every $n$. The mapping $\psi^{*\infty}$ is not injective, therefore, there are distinct points $x_1$ and $x_2$ from $\Gamma$ with the same image. We consider two cases.

First case, $\gamma(x_1)=T^k(\gamma(x_2))$ for some $k\in\ZZ$. Without loss of generality, we assume that $k\geq 0$. By minimality, there is a sequence $(n_i)_{i=0}^{+ \infty}$ of nonnegative integers such that $S^{n_i}(x_1)$ and $S^{n_i}(x_2)$ converges to the respective limits $x_3$ and $x_4$, where $x_3\not\in\gamma^{-1}(\mathcal O)$. Since $\gamma(x_3)=\gamma(S^k(x_3))$, we get $x_3=S^kx_4$. Hence, $\psi^{*\infty}(x_4)=\psi^{*\infty}(x_3)=\psi^{*\infty}(S^kx_4)=S^k(\psi^{*\infty}(x_4))$. Since the Sturmian subshift has no factor with a periodic point, we get $k=0$ and $\gamma(x_1)=\gamma(x_2)$. Hence, without loss of generality, we can assume that $x_1=x(n)$ and $x_2=y(n)$ for some $n\in\NN$. We get 
$$\psi^{*\infty}(x(0))=\psi^{*\infty}(S^n(x(n)))=\psi^{*\infty}(S^n(y(n)))=\psi^{*\infty}(y(0)).$$ 
The assumption that $\psi$ does not ignore the first letter implies that the images of $x(0)$ and $y(0)$ under the map $\psi^{*\infty}$ must differ. It is a contradiction.

Second case, $\gamma(x_1)-\gamma(x_2) \bmod{1}$ does not belong to the both-sided orbit of $0$. Let us define the mapping $\eta$ from $[0,1[$ to $(\Delta^\NN,S)$ as follows: $\eta(z)=\psi^{*\infty}(x(n))$ if $z=T^{-n}(0)$ for some $n\in\NN$, and $\eta(z)=\psi^{*\infty}(\gamma^{-1}(z))$ otherwise. The map $\eta$ is continuous, whenever for every $n\in\NN$, $\psi^{*\infty}(x(n))=\psi^{*\infty}(y(n))$. Given $n\in\NN$, there are sequences $(m_i)_{i=0}^{+\infty}$ and $(n_i)_{i=0}^{+\infty}$ such that $S^{m_i}(x_1)$ converges to $x(n)$ and $S^{n_i}(x_1)$ converges to $y(n)$. We can suppose that $S^{m_i}(x_2)$ and $S^{n_i}(x_2)$ converges to the limits $x_3$ and $x_4$, respectively. We get
$$\gamma(x_3)-\gamma(x(n))=\gamma(x_1)-\gamma(x_2)=\gamma(x_4)-\gamma(y(n)) \quad \text{mod } 1.$$
Hence, $\gamma(x_3)$ does not belong to the both-sided orbit of $0$. Moreover, $\gamma(x(n))=\gamma(y(n))$. It implies that $x_3$ and $x_4$ coincide and 
$$\psi^{*\infty}(x(n))=\psi^{*\infty}(x_3)=\psi^{*\infty}(x_4)=\psi^{*\infty}(y(n)).$$  
Hence, $\eta$ is a non-constant continuous map from the unit circle into a totally disconnected space. This is a contradiction. 
\end{proof}

Next proposition considers injectivity of a finite sliding block code and gives a positive answer to Question \ref{qes:1}. 
Having proved the previous proposition and Lemma \ref{lem:injective_infinite}, a natural strategy of the proof of Proposition \ref{pro:sturm-sliding} would be to prove that (\ref{item:2}) implies (\ref{item:1}) (last two conditions are surely equivalent). Unfortunately, we did not find an easy way to prove it in this way. Instead, we apply Proposition \ref{pro:precise-sturm} to prove that (\ref{item:4}) implies (\ref{item:1}). As a byproduct, we obtain another proof of Proposition \ref{pro:Durand}. However, this proof is more complicated because it involves all the machinery needed to prove Theorem \ref{thm:mainresult}.

\begin{pro}\label{pro:sturm-sliding}
  Let $(\Gamma,S)$ be a Sturmian subshift, $\psi^{*\infty}$ be the homomorphism from
  $(\Gamma,S)$ to a subshift $(\Delta^\NN,S)$ induced by a local rule
  $\psi:\lc^m(\Gamma) \to\Delta$.

  If the rule $\psi$ is minimal and does not ignore the first letter,
  then the following conditions are equivalent:
  \begin{enumerate}
  \item\label{item:1} $\psi^{*n}$ is injective for some natural number,
  \item\label{item:2} $\psi^{*\infty}$ is injective.
  \item\label{item:3} $\psi^{*\infty}$ is not constant,
  \item\label{item:4} $\psi$ is not constant.
  \end{enumerate}
\end{pro}

\begin{proof} It is readily seen that the conditions above are ordered
  from the strongest to the weakest,
  i.e. $(\ref{item:1})\Rightarrow(\ref{item:2})\Rightarrow(\ref{item:3})\Rightarrow(\ref{item:4})$. It
  suffices to prove that (\ref{item:4}) implies
  (\ref{item:1}). Suppose that $\psi:\lc^m(\Gamma)\to\Delta$ is a
  non-constant minimal local rule that does not ignore the first
  letter.

  For $a\in\Delta$, define $R_a=\left(\psi^{*\infty}\right)^{-1}
  [a]$. The set $\rc=\{R_a \mid a\in\Delta\}\setminus \{ \emptyset \}$
  forms a partition and for every $n\in\NN$, $u\in\Delta^n$, we get
  that the set $R_u$, defined in Section \ref{sec:intro}, satisfies
  the following condition,
  \begin{align*}
    R_{u}&=\bigcap^{n-1}_{k=0}S^{-k}\left(\left(\psi^{*\infty}\right)^{-1}
      [u_k]\right)=\bigcap^{n-1}_{k=0}\left(\psi^{*\infty}\right)^{-1}\left(
      T^{-k}[u_k]\right)=
    \left(\psi^{*\infty}\right)^{-1} [u]\\
    &=\bigcup\{[v] \subset\Gamma \mid v\in\lc^{m+n-1}(\Gamma),
    \psi^{*n}(v)=u\}.
  \end{align*}

  In particular, $R_u$ is a union of cylinders $[v]$ from $[\Sigma^m]$
  for every $u\in\Delta^*$. Hence, $\rc$ is rougher than
  $[\Sigma^m]$. Let $\ell$ be the largest and $m'$ the least possible
  integer such that $\rc$ is rougher than $[\Sigma^{m'}]_\ell$.

  Since $\psi$ does not ignore the first letter, there are two words
  $u,v\in\lc^m(\Gamma)$ such that $u[1,m)=v[1,m)$ and
  $\psi(u)\neq\psi(v)$. In particular, $u_0\neq v_0$. We get that $u$
  and $v$ are in the same set from $[\Sigma^{m-1}]_1$, but they are
  not from the same set from $\rc$. Thus, $\rc$ is not rougher than
  $[\Sigma^m]_1$. It implies that $\ell=0$. The minimality of the
  local rule implies that there are two words $u,v\in\lc^m(\Gamma)$
  such that $u[0,m-1)=v[0,m-1)$ and $\psi(u)\neq\psi(v)$. In
  particular, $u_{m-1}\neq v_{m-1}$. We get that $u$ and $v$ are in
  the same set from $[\Sigma^{m-1}]_0$, but they are not from the same
  set from $\rc$. Thus, $\rc$ is not rougher than
  $[\Sigma^{m-1}]_0$. It implies that $m'=m$.

  By Proposition \ref{pro:precise-sturm}, there exists $n\in\NN$ such
  that $\rc^n=[\Sigma^{m+n-1}]$. Hence, for $u\in\Delta^n$, $R_u$ is
  either empty, or equal to one cylinder from $[\Sigma^{m+n-1}]$. It
  implies that there is at most one $v\in\Sigma^{m+n-1}$ such that
  $\psi^{*n}(v)=u$. Thus, $\psi^{*n}$ is injective.
\end{proof}

Let us remark that the rule $\psi$ mentioned in Example \ref{exm:symbolic} can be restricted to a Sturmian subshift and thus seemingly produce a counterexample to Proposition \ref{pro:sturm-sliding}.
However, it is not a counterexample since the local rule applied to a Sturmian subshift has to be restricted to the words from the language and that means for the Sturmian subshift that either $00$ or $11$ is not be taken into account. But then, either the local rule $\psi$ is not minimal (when $00$ is not in the language), or it is injective (when $11$ is not in the language).

\section{Open problems}
\label{sec:open-problems}
The first problem concerns the rotation of the unit circle and the
evolution of a partition that consists of finite unions of
intervals. We proved that if the endpoints of the intervals belong to
the past trajectory of the point zero, then the refinements of the
partition will eventually consist of connected sets, i.e.,
intervals. The question is if it remains to be true if we omit the
assumption on the endpoints of the intervals. It is not difficult to
see that it is not true in full generality. The counterexample is the
partition $\rc$ into two sets $[0,1/4)\cup [1/2,3/4)$ and
$[1/4,1/2)\cup [3/4,1)$. The symmetry of the partition ensures that
for every $n\in\NN$ and every $x\in [0,1/4)$, there exist set $M$ and
$N$ from $\pc^n$ such that $M$ contains points $x$ and $x+1/2$ and $N$
contains the points $x+1/4$ and $x+3/4$. 
In particular, $M$ and $N$
are not connected. By the same argument we can show that the counterexample is any 
non-trivial partition $\rc$ that is invariant under a rational rotation, where the 
 invariance under a rational rotation means that there exist a natural number $m\geq 2$ such that for every $x\in[0,1)$, the number $(x+1/m)\mod \ZZ$ belongs to the same set from $\rc$ as $x$ does. 
As far as we know, it is not known whether $\rc^n$ eventually consists of connected sets (intervals) in the case when $\rc$ is not invariant under a rational rotation. 

The second problem is related to the main result formulated in terms of sliding block codes:
for which subshifts can one give a positive to Question \ref{qes:1}?
In section \ref{sec:block}, we showed that for a certain high and low subword complexity subshifts the answer is negative.
Does it mean that only Sturmian subshifts and its factors admit a positive answer?
Might it be another characteristic of these subshifts?
Or does there exist other subshifts with this property?
A good candidate might be another coding of the rotation of the unit circle or some class of substitution subshifts.

\section*{Acknowledgments}

We would like to thank Pierre Arnoux, Val\'erie Berth\'e, Gilles Didier and Fabien Durand for their useful remarks.
The second author would like to thank for the support of Czech Science Foundation grants GA\v
CR, grant no. 13-35273P.

\bibliographystyle{alpha}
\bibliography{biblio}

\end{document}